\documentclass[reqno,12pt]{amsart}

\usepackage{amsmath}
\usepackage{amsfonts}
\usepackage{amssymb}
\usepackage{eufrak}
\usepackage{amscd}
\usepackage{amsthm}
\usepackage{amstext}
\usepackage[all]{xy}

\newcommand{\Ad}{\operatorname{Ad}}

\newcommand{\Ker}{\operatorname{Ker}}
\newcommand{\id}{\operatorname{id}}

\newcommand{\ev}{\operatorname{ev}}

   \theoremstyle{plain}
   \newtheorem{thm}{Theorem}[section]
   \newtheorem{prop}[thm]{Proposition}
   \newtheorem{lem}[thm]{Lemma}
   \newtheorem{cor}[thm]{Corollary}
   \theoremstyle{definition}

   \theoremstyle{remark}

\author{V. Manuilov}

\date{}

\address{Moscow State University,
Leninskie Gory, Moscow, 
119991, Russia}

\email{manuilov@mech.math.msu.su}

\thanks{The author acknowledges partial support by the RFBR grant No. 16-01-00357.}

\title{A noncommutative version of Farber's topological complexity}

\begin{document}
\maketitle

\begin{abstract}
Topological complexity for spaces was introduced by M. Farber as a minimal number of continuity domains for motion planning algorithms. It turns out that this notion can be extended to the case of not necessarily commutative $C^*$-algebras. Topological complexity for spaces is closely related to the Lusternik--Schnirelmann category, for which we do not know any noncommutative extension, so there is no hope to generalize the known estimation methods, but we are able to evaluate the topological complexity for some very simple examples of noncommutative $C^*$-algebras.  

\end{abstract}

\section*{Introduction}

Gelfand duality between compact Hausdorff spaces and unital commutative $C^*$--algebras allows to translate some topological constructions and invariants into the noncommutative setting. The most successful example is $K$-theory, which became a very useful tool in $C^*$-algebra theory. Homotopies between $*$-homomorphisms of $C^*$-algebras also play an important role, but there is no nice general homotopy theory for $C^*$-algebras due to the fact that the loop functor has no left adjoint \cite {Uuye}, Appendix A. Nevertheless, there are some homotopy invariants that allow noncommutative versions.   

The aim of our work is to show that M. Farber's topological complexity \cite{Farber} is one of those. In Section \ref{Section1} we recall the original commutative definition of topological complexity, and in Section \ref{Section2} we use Gelfand duality to reverse arrows in this definition, and show that the resulting noncommutative definition generalizes the commutative one. In the remaining two sections we calculate topological complexity for some simple examples of $C^*$-algebras. In particular, we show that introducing noncommutative coefficients may decrease topological complexity. Although in most our examples topological complexity is either 1 or $\infty$, we provide a noncommutative example
with topological complexity 2.

The author is grateful to A. Korchagin for helpful comments.

\section{Farber's topological complexity}\label{Section1}

The topological approach to the robot motion planning problem was initiated by M. Farber in \cite{Farber}. Let us recall his basic construction. Let $X$ be the configuration space of a mechanical system. A continuous path $\gamma:[0,1]\to X$ represents a motion of the system, with $\gamma(0)$ and $\gamma(1)$ being the initial and the final state of the system. If $X$ is path-connected then the system can be moved to an arbitrary state from a given state. Let $PX$ denote the space of paths in $X$ with the compact-open topology, and let 
\begin{equation}\label{pi}
\pi:PX\to X\times X
\end{equation}
be the map given by $\pi(\gamma)=(\gamma(0),\gamma(1))$. A continuous {\it motion planning algorithm} is a continuous section $$
s:X\times X\to PX
$$ 
of $\pi$. Typically, there may be no continuous motion planning algorithm, so one may take a covering of $X\times X$ by sets $V_1,\ldots,V_n$ (domains of continuity) and require existence of continuous sections 
$$
s_i:V_i\to PX|_{V_i}
$$ 
of maps $\pi_i:PX|_{V_i}\to V_i$, $i=1,\ldots,n$. Here $PX|_{V_i}$ denotes the restiction of $\pi$ onto $V_i$, i.e. the subset of paths $\gamma:[0,1]\to X$ such that $(\gamma(0),\gamma(1))\in V_i$. In this case, the collection of the sections $s_i$, $i=1,\ldots,n$, is called a (discontinuous) motion planning algorithm. There are several versions of the definition, which use various kinds of coverings, e.g. coverings by open or closed sets, or by Euclidean neighborhood retracts, etc., but most of them agree on simplicial polyhedra (cf. \cite{Farber_survey}, Theorem 13.1). The topological complexity $TC(X)$ of $X$ is the minimal number $n$ of domains of continuity, i.e. the minimal number $n$, for which there exists a covering $V_1,\ldots,V_n$ and continuous sections $s_i$ as above. This number measures the complexity of the problem of navigation in $X$.   

\section{Noncommutative version of topological complexity}\label{Section2}

For a compact Hausdorff space $X$ we can rewrite the above construction in terms of unital commutative $C^*$-algebras and their unital $*$-homomorphisms using Gelfand duality. Let $C(X)$ denote the commutative $C^*$-algebra of complex-valued continuous functions on $X$. A closed covering $V_1,\ldots,V_n$ of $X\times X$ corresponds to $n$ surjective $*$-homomorphisms 
$$
j_i:C(X)\otimes C(X)\to C(V_i),
$$ 
$i=1,\ldots,n$, with $\cap_{i=1}^n\Ker j_i=\{0\}$. As the path space $PX$ is not locally compact, it is not Gelfand dual to any $C^*$-algebra, but we can bypass this, replacing the sections $s_i$ by $*$-homomorphisms 
$$
\sigma_i:C(X)\to C(V_i)\otimes C[0,1]
$$ 
defined by
$$
\sigma_i(f)(x,t)=f(s_i(x)(t)),
$$
where $x\in V_i$, $t\in[0,1]$, $f\in C(X)$.
Let us denote by $\ev_t$ the $*$-homomorphism of evaluation at $t\in[0,1]$, and let us consider the compositions 
$$
\ev_0\circ\sigma_i, \ \ev_1\circ\sigma_i:C(X)\to C(V_i).
$$ 
Let 
$$
\pi_0,\pi_1:X\times X\to X
$$ 
denote the projections onto the first and the second copy of $X$ respectively, and let 
$$
p_0,p_1:C(X)\to C(X)\otimes C(X)
$$ 
be the corresponding $*$-homomorphisms. The condition $\pi\circ s_i=\id_{V_i}$ can be written as $\pi_k\circ\pi\circ s_i=\pi_k: V_i\to X$, $k=0,1$, which allows rewriting, in terms of $C^*$-algebras and $*$-homomorphisms, as $j_i\circ p_0=\ev_0\circ\sigma_i$, $j_i\circ p_1=\ev_1\circ\sigma_i$. Thus we have 

\begin{lem}\label{L1}
Continuous sections $s_i:V_i\to PX|_{V_i}$ exist iff there exist $*$-homomorphisms $\sigma_i$ making the diagrams 
\begin{equation}\label{diagram-definition}
\begin{xymatrix}{
C(X) \ar[r]^-{p_k}\ar[d]_-{\sigma_i} &C(X)\otimes C(X)\ar[d]^-{j_i}\\
C(V_i)\otimes C[0,1] \ar[r]_-{\ev_k}  & C(V_i),}
\end{xymatrix}
\end{equation}
$k=0,1$, commute.

\end{lem}  

Thus, we may define the topological complexity $TC(A)$ for a unital $C^*$-algebra $A$ as the minimal number $n$ of quotient $C^*$-algebras $B_1,\ldots,B_n$ of $A\otimes A$ with the quotient maps $q_i:A\otimes A\to B_i$, such that 
\begin{enumerate}
\item
$\cap_{i=1}^n\Ker q_i=\{0\}$;
\item
there exist $*$-homomorphisms 
$$
\sigma_i:A\to B_i\otimes C[0,1],\qquad i=1,\ldots,n,
$$
making the diagrams
\begin{equation}\label{diagram}
\begin{xymatrix}{
A \ar[r]^-{p_k}\ar[d]_-{\sigma_i} &A\otimes A\ar[d]^-{q_i}\\
B_i\otimes C[0,1] \ar[r]_-{\ev_k}  & B_i,}
\end{xymatrix}
\end{equation}
$k=0,1$, commute for each $i=1,\ldots,n$, where $p_0(a)=a\otimes 1$, $p_1(a)=1\otimes a$, $a\in A$.
\end{enumerate}
Here and further we always use $\otimes$ to denote the {\it minimal} tensor product of $C^*$-algebras.
If there is no such $n$ then we set $TC(A)=\infty$. 

\begin{cor}
For a compact Hausdorff space $X$, one has $TC(C(X))=TC(X)$ if $TC(X)$ is defined using {\sl closed coverings}.

\end{cor}

\begin{proof}
Commutativity of $A=C(X)$, hence of $A\otimes A$, implies commutativity of $B_i$, hence $B_i=C(V_i)$ for some $V_i$. Surjectivity of $q_i$ implies that $V_i$ is a closed subset of $X\times X$. The condition $\cap_{i=1}^n\Ker q_i=\{0\}$ means that $\{V_1,\ldots,V_n\}$ is a covering for $X\times X$. 

\end{proof}

As we shall see later, topological complexity is not well suited for general $C^*$-algebras, e.g. it is infinite for topologically non-trivial simple $C^*$-algebras, but there are two good classes of $C^*$-algebras, for which this characterization may be interesting --- the class of noncommutative CW complexes introduced in \cite{ELP} and the class of $C(X)$-algebras. Most of our examples are from the first class.

Note that in the commutative case, topological complexity makes sense only for path-connected spaces --- otherwise any two points may be not connected by a path, i.e. the map (\ref{pi}) is not surjective. There is no good $C^*$-algebraic analog for that, but the following holds:
\begin{lem}\label{L6}
Let $A=A_1\oplus A_2$. Then $TC(A)=\infty$.

\end{lem}
\begin{proof}
One has $A\otimes A=\oplus_{k,l=1}^2 A_k\otimes A_l$.  Let $q_i:A\otimes A\to B_i$, $i=1,\ldots,n$, and $\sigma:A\to B_i\otimes C[0,1]$ be as in the definition of topological complexity, and let $e_1=q_i(1_{A_1}\otimes 1_{A_1})$, $e_2=q_i(1_{A_1}\otimes 1_{A_2})$, $e_3=q_i(1_{A_2}\otimes 1_{A_1})$, $e_4=q_i(1_{A_2}\otimes 1_{A_2})$. Then $e_1,\ldots,e_4$ are projections in $B_i$, and, as $q_i$ is surjective, any element of $B_i$ has the form $\sum_{k=1}^4 e_kbe_k$. In particular, if $e\in B_i$ is a projection then each $e_kee_k$ is a projection, and if $e(t)$, $t\in [0,1]$, is a homotopy of projections, then we have four homotopies $e_ke(t)e_k$.

Let $a=1_{A_1}\oplus 0_{A_2}\in A$. Then $q_i(p_0(a))=e_1+e_2$ and $q_i(p_1(a))=e_1+e_3$ should be connected by a homotopy. This is possible only if $e_2=e_3=0$. As this argument does not depend on $i$, we conclude that $1_{A_1}\otimes 1_{A_2},1_{A_2}\otimes 1_{A_1}\in \cap_{i=1}^n\Ker q_i=\{0\}$ --- a contradiction.

\end{proof}

The topological complexity of a space $X$ can be estimated from above by using covering dimension of $X$, and from below using multiplicative structure in cohomology. Regretfully, these estimates cannot work in the noncommutative case, thus making the problem of evaluating topological complexity even more difficult.

\section{Case $TC(A)=1$}\label{Section3}

The condition $TC(A)=1$ means that the two inclusions of $A$ into $A\otimes A$, $p_0:a\mapsto a\otimes 1$ and $p_1:a\mapsto 1\otimes a$, are homotopic. This property is similar to, but different from that of approximately inner half flip \cite{Toms-Winter}, which means that $p_0$ and $p_1$ are approximately unitarily equivalent, i.e. there exist unitaries $u_n\in A\otimes A$ such that $\lim_{n\to\infty}\|p_1(a)-\Ad_{u_n}p_0(a)\|=0$ for any $a\in A$.

The condition $TC(A)=1$ imposes restrictions on the $K$-theory groups of $A$. Let $K_*(A)$ denote the graded $K$-theory group of $A$, and let $\mathbf 1\in K_0(A)$ be the class of the unit element. Recall that if $A$ is in the bootstrap class \cite{Schochet} then it satisfies the K\"unneth formula, hence $K_*(A)\otimes K_*(A)\subset K_*(A\otimes A)$. The bootstrap class is the smallest class which contains all separable type I $C^*$-algebras and is closed under extensions, strong Morita equivalence, inductive limits, and crossed products by $\mathbb R$ and by $\mathbb Z$.

\begin{lem}\label{K}
Let $A$ satisfy $K_*(A)\otimes K_*(A)\subset K_*(A\otimes A)$. If $K_*(A)\otimes\mathbf 1\neq \mathbf 1\otimes K_*(A)$ then $TC(A)>1$.

\end{lem}
\begin{proof}
This follows from homotopy invariance of $K$-theory groups. If $TC(A)=1$ then the flip on $K_*(A\otimes A)$ must induce the identity map.

\end{proof}

For spaces, it is known that $TC(X)=1$ iff $X$ is contractible. For $C^*$-algebras, it is reasonable to call a unital $C^*$-algebra $A$ {\it contractible to a point} if there exists a $*$-homomorphism $h:A\to A\otimes C[0,1]$ and a $*$-homomorphism $i:A\to\mathbb C$ such that $\ev_1\circ h=\id_A$ and $\ev_0\circ h=j\circ i$, where $j:\mathbb C\to A$ is defined by $j(1)=1_A$. If $B$ is a non-unital contractible $C^*$-algebra then its unitalization $B^+$ is contractible to a point.

\begin{lem}\label{L2}
Let $A$ be contractible to a point. Then $TC(A)=1$.

\end{lem}
\begin{proof}
Let $\alpha:A\otimes A\otimes C[0,1]$ be the flip, $\alpha(a_1\otimes a_2\otimes f)=a_2\otimes a_1\otimes f$, where $a_1,a_2\in A$, $f\in C[0,1]$. 
Let $h:A\to A\otimes C[0,1]$ be the homotopy as above. We write $h_t$ for $\ev_t\circ h$.

Define a $*$-homomorphism
$$
\sigma:A\to A\otimes A\otimes C[-1,1]
$$
by setting, for $a\in A$,  
$$
\sigma(a)(t)=\left\lbrace\begin{array}{cl}
\alpha(1\otimes h_t(a)),&{\mbox{if\ }}t\in[0,1];\\
1\otimes h_{-t}(a),&{\mbox{if\ }}t\in[-1,0].
\end{array}\right. 
$$ 
Then $\ev_1\circ\sigma(a)=a\otimes 1$, $\ev_{-1}\circ\sigma(a)=1\otimes a$. Continuity of $\sigma$ at $t=0$ follows from the equality $i(a)\otimes 1=1\otimes i(a)$.

\end{proof}

\begin{cor}
If $A_n=\{f\in C([0,1];M_n): f(1) \mbox{\ is\ scalar}\}$ then $TC(A_n)=1$.
\end{cor}

\begin{lem}\label{L3}
Let $TC(A)=1$. If there exists a unital $*$-homomorphism $i:A\to\mathbb C$ then $A$ is contractible to a point.

\end{lem}
\begin{proof}
Let $\sigma:A\to A\otimes A\otimes C[0,1]$ satisfy $\ev_0\circ\sigma(a)=a\otimes 1$ and $\ev_1\circ\sigma(a)=1\otimes a$, $a\in A$. Let $\bar\iota:A\otimes A\otimes C[0,1]\to A\otimes C[0,1]$ be the map defined by $\bar\iota(a_1\otimes a_2\otimes f)=i(a_1)\cdot a_2\otimes f$, where $a_1,a_2\in A$, $f\in C[0,1]$. Set $h=\bar\iota\circ\sigma:A\to A\otimes C[0,1]$. Then
$\ev_0\circ h(a)=i(a)\cdot 1$, $\ev_1\circ h(a)=a$, hence $h$ is the required homotopy.

\end{proof}

Below we list three examples of $C^*$-algebras with topological complexity 1. The proofs are known to specialists, but we could not find exact references.

\begin{prop}\label{L3}
One has $TC(M_n)=1$. 

\end{prop}
\begin{proof}
Let $U$ be a unitary in $M_{n^2}\cong M_n\otimes M_n$ such that $\Ad_U$ is an automorphism of $M_{n^2}$ that interchanges $M_n\otimes 1$ with $1\otimes M_n$. If $M_n$ acts on an $n$-dimensional space $H_n$ with the orthonormal basis $\{e_i\}_{i=1}^n$ then $U$ interchanges vectors $e_i\otimes e_j$ and $e_j\otimes e_i$ when $i\neq j$. Let $U_t$, $t\in[0,1]$, be the path connecting $U$ with 1 constructed using the standard rotation formula. Define $\sigma:M_n\to M_n\otimes M_n\otimes C[0,1]$ by $\sigma(a)(t)=\Ad_{U_t}(a\otimes 1)$, $a\in M_n$. 

\end{proof}

The above example can be extended to UHF algebras:
\begin{prop}\label{L4}
If $A$ is a UHF algebra then $TC(A)=1$. 

\end{prop}
\begin{proof}
Let $n,k$ be integers, $\varphi:M_n\to M_{kn}$ a unital $*$-homomorphism. Let $\sigma':M_n\to M_n\otimes M_n\otimes C[0,1]$ and $\sigma'':M_{kn}\to M_{kn}\otimes M_{kn}\otimes C[0,1]$ be the maps constructed in the proof of Lemma \ref{L3},
$\sigma'(a')(t)=\Ad_{U'_t}(a'\otimes 1)$, $\sigma''(a'')(t)=\Ad_{U''_t}(a''\otimes 1)$, $a'\in M_n$, $a''\in M_{kn}$. 

Then the diagram  
\begin{equation}\label{333}
\begin{xymatrix}{
M_n \ar[r]^-{\sigma'}\ar[d]_-{\varphi} &M_n\otimes M_n\otimes C[0,1]\ar[d]^-{\varphi\otimes\varphi\otimes\id}\\
M_{kn}\ar[r]_-{\sigma''}  & M_{kn}\otimes M_{kn}\otimes C[0,1]}
\end{xymatrix}
\end{equation}
commutes. Let $A$ be the direct limit of matrix algebras $A_n=M_{m_n}$, where $m_n$ divides $m_{n+1}$, $n\in\mathbb N$. Commutativity of the diagram (\ref{333}) shows that the maps $\sigma^{(n)}:A_n\to A_n\otimes A_n\otimes C[0,1]$ agree, so, for any $t\in[0,1]$ one can define the limit map $\sigma_t:A\to A\otimes A$ such that $(\sigma_t)|_{A_n}=\ev_t\circ\sigma^{(n)}$. Since $\|\sigma_t(a)\|\leq\|a\|$ for any $a\in A$ and any $t\in[0,1]$, continuity of $\sigma_t(a)$ with respect to $t$ for $a\in\cup_{n=1}^\infty A_n$ implies continuity of $\sigma_t(a)$ for any $a\in A$. This means that the family $\{\sigma_t\}_{t\in[0,1]}$ defines a $*$-homomorphism $\sigma:A\to A\otimes A\otimes C[0,1]$, which provides the required homotopy.

\end{proof}

Let $\mathcal O_2$ be the Cuntz algebra generated by two isometries $s_1$, $s_2$ satisfying $s_1s_1^*+s_2s_2^*=1$. 

\begin{prop}
One has $TC(\mathcal O_2)=1$.

\end{prop}
\begin{proof}
Let $u=s_1^*\otimes s_1+s_2^*\otimes s_2\in\mathcal O_2\otimes\mathcal O_2$. It is unitary, and it suffices to check on generators that $p_0=\Ad_u p_1$ (cf. \cite{Rordam}, Theorem 5.1.2). But $\mathcal O_2\otimes\mathcal O_2\cong\mathcal O_2$ (\cite{Rordam}, Theorem 5.2.1), and the unitary group of $\mathcal O_2$ is contractible, hence, $p_0$ and $p_1$ are homotopic.

\end{proof}

\section{General case}\label{Section4}

Let $\mathbb K^+$ be the unitalized algebra of compact operators. In contrast with Lemma \ref{L3}, its topological complexity is infinite. This often happens for $C^*$-algebras with few ideals.

\begin{lem}
One has $TC(\mathbb K^+)=\infty$.

\end{lem}
\begin{proof}
Let $B_1,\ldots, B_n$ be the quotients of $\mathbb K^+\otimes\mathbb K^+$. If they satisfy the definition of topological complexity then one of them must coincide with $\mathbb K^+\otimes\mathbb K^+$ itself, in which case other quotients are redundant. Therefore, if $TC(\mathbb K^+)\neq\infty$ then $TC(\mathbb K^+)=1$. To show that this is not the case, recall that $K_0(\mathbb K^+)\cong \mathbb Z^2$ and use Lemma \ref{K}.

\end{proof}

\begin{lem}\label{L5}
Let $TC(A)>1$. If $A$ is simple then $TC(A)=\infty$.

\end{lem}
\begin{proof}
It follows from \cite{Takesaki} that $A\otimes A$ is simple, hence any possible quotient $B$ must equal $A\otimes A$.

\end{proof}

It follows that topological complexity distinguishes commutative $C^*$-algebras from their non-commutative deformations. 
For example, consider an irrational rotation algebra $A_\theta$, $\theta\in[0,1]\setminus\mathbb Q$, often called a non-commutative torus. It is simple and has the same $K$-theory as the usual torus $\mathbb T^2$ \cite{Davidson}, hence $TC(A_\theta)=\infty$, while for a usual torus $\mathbb T^2$ one has $TC(C(\mathbb T^2))=3$ (cf. \cite{Farber_survey}, Example 16.4). 

Nevertheless, tensoring by matrices does not increase topological complexity.

\begin{prop}\label{L9}
For any compact Hausdorff space $X$, one has $TC(C(X)\otimes M_n)\leq TC(C(X))$.

\end{prop}
\begin{proof}
Let $TC(C(X))=k$, and let $q_i:C(X)\otimes C(X)\to B_i$ and $\sigma_i:C(X)\to B_i\otimes C[0,1]$, $i=1,\ldots,k$, be as in the definition of topological complexity. Set $\overline{B}_i=B_i\otimes M_n\otimes M_n$, $\overline{q}_i=q_i\otimes\id:C(X)\otimes C(X)\otimes M_n\otimes M_n\to \overline{B}_i$. Define $\overline{\sigma}_i:C(X)\otimes M_n\to \overline{B}_i\otimes C[0,1]$ by $\overline{\sigma}_i(f\otimes m)(t)=\sigma_i(f)\otimes \Ad_{U_t}(m\otimes 1)\in B_i\otimes C[0,1]\otimes M_n\otimes M_n$, $f\in C(X)$, $m\in M_n$, $t\in[0,1]$, and $U_t$ as in the proof of Lemma \ref{L3}. Then the maps $\overline{q}_i$, $\overline{\sigma}_i$ make the corresponding diagrams commute, hence $TC(C(X)\otimes M_n)\leq TC(C(X))$. 

\end{proof}

More generally, one has
\begin{prop}
Let $TC(A)=n$, $TC(C)=m$. Then $TC(A\otimes C)\leq nm$.

\end{prop}
\begin{proof}
Let $q^A_i:A\otimes A\to B_i$, $\sigma^A_i:A\to B_i\otimes C[0,1]$, $i=1,\ldots,n$, and $q^C_j:C\otimes C\to D_j$, $\sigma^C_j:C\to D_j\otimes C[0,1]$, $j=1,\ldots,m$, be as in the definition of topological complexity. Let $\Delta:C([0,1]^2)\to C[0,1]$ be the map induced by the diagonal embedding $[0,1]\to[0,1]^2$ and define the composition 

$$
\begin{xymatrix}{
\sigma_{ij}:A\otimes C\ar[rr]^-{\sigma^A_i\otimes\sigma^C_j}&&B_i\otimes D_j\otimes C([0,1]^2)\ar[rr]^-{\id\otimes\Delta}&&B_i\otimes D_j\otimes C[0,1].
}\end{xymatrix}
$$

Then the diagram 
$$
\begin{xymatrix}{
A\otimes C \ar[rr]^-{p^A_k\otimes p^C_k}\ar[d]_-{\sigma_{ij}} &&A\otimes C\otimes A\otimes C\ar[d]^-{q^A_i\otimes q^C_j}\\
B_i\otimes D_j\otimes C[0,1] \ar[rr]_-{\ev_k}  && B_i\otimes D_j,}
\end{xymatrix}
$$
$k=0,1$, commutes for all $i$, $j$.

\end{proof}

Remark that in the commutative case the tensor product of $C^*$-algebras is Gelfand dual to the product of spaces, and 
there is a much better estimate $TC(A\otimes C)\leq n+m-1$ (\cite{Farber}, Theorem 11).

We have no examples with $TC(C(X)\otimes M_n)<TC(C(X))$, but tensoring by a more general $C^*$-algebra may decrease topological complexity. Let $U(A)$ denote the group of unitaries of a $C^*$-algebra $A$. Recall that $U(\mathcal O_2)$ is contractible \cite{Thomsen}.

Let $\mathbb S$ denote the circle. It is known that $TC(C(\mathbb S))=2$. 

\begin{thm}
Let $A$ satisfy $TC(A)=1$, $\pi_0(U(A))=\pi_1(U(A))=0$ $($e.g. $A=\mathcal O_2$$)$. Then $TC(C(\mathbb S)\otimes A)=1$.

\end{thm}
\begin{proof}
We have to connect by a homotopy the two $*$-homomorphisms $\sigma_i:C(\mathbb S)\otimes A\to C(\mathbb S)\otimes A\otimes C(\mathbb S)\otimes A$, $i=0,1$, given by $\sigma_0(f\otimes a)=f\otimes a\otimes 1\otimes 1$ and $\sigma_1(f\otimes a)=1\otimes 1\otimes f\otimes a$, $f\in C(\mathbb S)$, $a\in A$. Note that these maps are determined by their values on $u\otimes a$, where $u(x)=e^{2\pi ix}$, $u\in C(\mathbb S)$. By assumption, any unitary in $C(\mathbb S)\otimes A$ has a homotopy that connects it with $1\otimes 1$. Let $u_t$, $t\in[2/3,1]$, be a homotopy, in the unitary group of $C(\mathbb S)\otimes A$, that connects $u\otimes 1$ with $1\otimes 1$. Then the homotopy $\sigma_t$, given by $\sigma_t(u\otimes a)=1\otimes u_t\otimes a$ connects $\sigma_1$ with $\sigma_{2/3}$ given by $\sigma_{2/3}(u\otimes a)=1\otimes 1\otimes 1\otimes a$. Similarly, one can connect $\sigma_0$ with $\sigma_{1/3}$ given by $\sigma_{1/3}(u\otimes a)=1\otimes a\otimes 1\otimes 1$. Finally, as $TC(A)=1$, $\sigma_{1/3}$ and $\sigma_{2/3}$ are homotopic.

\end{proof}

Our next examples show how sensitive topological complexity may be. Let 
$$
A_2=\{f\in C([0,1];M_2):f(1)\mbox{\ is\ diagonal}\}.
$$
This algebra is considered as a noncommutative version of the non-Hausdorff $T_1$ space $X_2$ obtained from two intervals $\{(x,y)\in[0,1]^2:y=0\mbox{\ or\ 1}\}$ by identifying the points $(x,0)$ and $(x,1)$ for each $x\in[0,1)$ \cite{Connes}. Although $X_2$ is not Hausdorff, it is contractible, hence $TC(X_2)=1$.

\begin{lem}\label{L7}
One has $TC(A_2)=\infty$.

\end{lem}
\begin{proof}
Suppose that $TC(A_2)=n<\infty$. Let $q_i:A_2\otimes A_2\to B_i$, $i=1,\ldots,n$, be as in the definition of topological complexity. There are two $*$-homomorphisms from $A_2$ to $\mathbb C$, given by $r_0(f)=f_{11}(1)$ and $r_1(f)=f_{22}(1)$, where $f\in A_2$. It is easy to see that each quotient map from $A_2$ factorizes through the restriction map on
a closed subset of $[0,1]^2$. As $\cap_{i=1}^n\Ker q_i=\{0\}$, there is at least one $i$ such that $r_0\otimes r_1$ factorizes through $q_i$. Further, we may argue as in Lemma \ref{L6}: the maps $(r_0\otimes r_1)\circ p_0$ and $(r_0\otimes r_1)\circ p_1$ from $A_2$ to $\mathbb C$ shoud be homotopic. Let $a=\left(\begin{smallmatrix}1&0\\0&0\end{smallmatrix}\right)\in A_2$. Then $(r_0\otimes r_1)\circ p_0(a)=1$, $(r_0\otimes r_1)\circ p_1(a)=0$, which makes homotopy between $(r_0\otimes r_1)\circ p_0$ and $(r_0\otimes r_1)\circ p_1$ impossible.

\end{proof}

Let 
$$
D_n=\{f\in C([0,1]; M_n): f(0),f(1)\mbox{\ are\ scalars}\}
$$ 
be a (unital) dimension-drop algebra. 

\begin{lem}\label{DDT}
If $n>1$ then $TC(D_n)=\infty$.

\end{lem}
\begin{proof}
We identify $D_n\otimes D_n$ with the subalgebra of functions $f=f(x,y)$ in $C([0,1]^2;M_n\otimes M_n)$ satisfying the obvious boundary conditions.
As above, if there exist $k$ quotients $B_1\ldots,B_k$ of $D_n\otimes D_n$ then at least one of them surjects onto a copy of $\mathbb C$ that identifies with restrictions of functions $f$ onto the point $(1,0)\in[0,1]^2$. Denote this map by $\mu:B_{i_0}\to\mathbb C$. If there is a homotopy $\sigma_{i_0}:D_n\to B\otimes C[0,1]$ then it restricts to a homotopy $D_n\to\mathbb C\otimes C[0,1]$. If the diagram (\ref{diagram}) commutes then $\mu\circ\ev_0\circ\sigma_{i_0}(f)=f(1)$ and $\mu\circ\ev_1\circ\sigma_{i_0}(f)=f(0)$, $f\in D_n$. But these two maps are not homotopic. 

\end{proof}
In both examples, $TC$ infinite means that there is no ``path'' connecting 0 and 1 in the noncommutative versions of an interval. 
In contrast with these examples is our next one.
Let 
$$
S_n=\{f\in C([0,1]; M_n): f(0)=f(1)\mbox{\ is\ scalar}\}.
$$ 
This is an algebra of matrix-valued functions on a circle, with the dimension drop at one point. If $n=1$ then $S_1$ is exactly the algebra of continuous functions on a circle.

\begin{thm}
For any $n\in\mathbb N$, $TC(S_n)=2$.

\end{thm}
\begin{proof}
We identify $S_n\otimes S_n$ with the algebra of $M_n\otimes M_n$-valued functions on $[0,1]^2$ with obvious boundary conditions. Let 
$$
Y_1=\{(x,y)\in[0,1]^2:|x-y|\leq 2/3\},
$$
$$
Y_2=\{(x,y)\in[0,1]^2:x\geq 2/3, y\leq 1/3\}\cup\{(x,y)\in[0,1]^2:x\leq 1/3, y\geq 2/3\}.
$$ 
Then $Y_1\cup Y_2=[0,1]^2$. Let $B_i$, $i=1,2$, be the algebras of continuous $M_n\otimes M_n$-valued functions with the same boundary conditions as in $S_n\otimes S_n$, and let $q_i:S_n\otimes S_n\to B_i$ be the quotient $*$-homomorphisms induced by restrictions onto $Y_i$.

We have to construct homotopies $\sigma_i:S_n\to B_i\otimes C[0,1]$ such that 
\begin{equation}\label{kontsy}
\ev_0\circ\sigma_i(f)(x,y)=f(x)\otimes 1,\quad \ev_1\circ\sigma_i(f)=1\otimes f(y).
\end{equation}
For $i=1$, $\ev_0\circ\sigma_1$ is homotopic to $\sigma'$ defined by 
$$
\sigma'(f)(x,y)=\left\lbrace\begin{array}{cl}
f(0)\otimes 1,&\mbox{for\ }x+y\geq 4/3\mbox{\ or\ }x+y\leq 2/3;\\f(\frac{x+y}{2/3}-1)\otimes 1,&\mbox{for\ }2/3\leq x+y\leq 4/3.\end{array}\right.
$$
Similarly, $\ev_1\circ\sigma_1$ is homotopic to $\sigma''=\Ad_U(\sigma')$, where $U$ intertwines $M_n\otimes 1$ and $1\otimes M_n$. Finally, $\sigma'$ is homotopic to $\sigma''$, as $\Ad_{U_t}$ maps scalars into scalars for any $t$, where $U_t$ is a path connecting $U$ with 1, so the boundary conditions on $Y_1$ hold.

For $i=2$, as 
$$
\{(x,y)\in[0,1]^2:x\geq 2/3, y\leq 1/3\}\cap\{(x,y)\in[0,1]^2:x\leq 1/3, y\geq 2/3\}=\emptyset,
$$ 
so after identifying 0 and 1, there is a single common point $(0,1)=(1,0)$. That's why we can construct the required homotopy separately for each of the $C^*$-algebras corresponding to these sets, but with the additional requirement that the two homotopies should agree at this common point. And as these sets are symmetric, it suffices to construct a homotopy for only one of them. Let $B_0$ denote the $C^*$-algebra of $M_{n^2}$-valued functions on 
$$
\{(x,y)\in[0,1]^2:x\geq 2/3, y\leq 1/3\}
$$ 
with the obvious boundary conditions, and let $q_0:S_n\otimes S_n\to B_0$ be the restriction quotient map. 

Note that the maps $\ev_0\circ\sigma_0$ and $\ev_1\circ\sigma_0$ (\ref{kontsy}) factorize through $A_0$ and $A_1$ respectively, where 
$$
A_0=\{f\in C([2/3,1];M_n):f(1)\mbox{\ is\ scalar}\},
$$ 
$$
A_1=\{f\in C([0,1/3];M_n):f(0)\mbox{\ is\ scalar}\}
$$ 
(i.e. with no restrictions at one of the end-points), hence the map $\ev_0\circ\sigma_0$ is homotopic to $\sigma'_0$ given by $$
\sigma'_0(f)(x,y)=f(1)\otimes 1,
$$ 
and the map $\ev_1\circ\sigma_0$ is homotopic to $\sigma''_0$ given by 
$$
\sigma'_0(f)(x,y)=1\otimes f(0).
$$ 
But, as $f(0)=f(1)$, they are homotopic. Along all these homotopies, their values at the point $(1,0)$ are the same.  
Thus, $TC(S_n)\leq 2$. 

To show that $TC(S_n)\neq 1$, let us calculate its $K$-theory groups. As $S_n$ is a split extension of $\mathbb C$ by the suspension $SM_n$ over $M_n$, one has $K_0(S_n)\cong K_1(S_n)\cong\mathbb Z$. Then 
$$
K_1(S_n\otimes S_n)\cong K_0(S_n)\otimes K_1(S_n)\oplus K_1(S_n)\otimes K_0(S_n).
$$ 
Let $(p_k)_*:K_1(S_n)\to K_1(S_n\otimes S_n)$ be the maps induced by the $*$-homomorphisms $p_k:S_n\to S_n\otimes S_n$, $k=0,1$, and let $e$ and $u$ be generators for $K_0(S_n)$ and for $K_1(S_n)$ respectively. Then 
$$
(p_0)_*(u)=u\otimes e\in K_1(S_n)\otimes K_0(S_n)\subset K_1(S_n\otimes S_n),
$$ 
$$
(p_1)_*(u)=e\otimes u\in K_0(S_n)\otimes K_1(S_n)\subset K_1(S_n\otimes S_n).
$$ 
As these elements are different,  there is no homotopy that connects $p_0$ with $p_1$.

\end{proof}

%


\end{document}